\documentclass[12pt]{amsart}
\pdfoutput=1

\usepackage{amsmath,amssymb,amsthm}
\usepackage{mathtools}
\usepackage{thm-restate}
\usepackage{tikz}
\usetikzlibrary{fadings}
\usepackage{comment}
\usepackage{enumitem}

\usepackage{graphicx}              
\usepackage{caption}
\usepackage{subcaption}
\usepackage{graphics}

\usepackage{tikz}
\usetikzlibrary{arrows}
\tikzset{edge/.style = {->,> = stealth'}}
\usetikzlibrary{angles,quotes}
\usetikzlibrary{decorations.pathreplacing}
\definecolor{mygray}{RGB}{160,160,160}

\usepackage{xcolor}
\definecolor{CornflowerBlue}{rgb}{0.39, 0.58, 0.93}
\definecolor{LavenderMagenta}{rgb}{0.93, 0.51, 0.93}
\definecolor{PastelOrange}{rgb}{1.0, 0.7, 0.28}

\usepackage[unicode]{hyperref}
\hypersetup{
	colorlinks,
	linkcolor={red!60!black},
	citecolor={green!60!black},
	urlcolor={blue!60!black},
}
\usepackage[abbrev, msc-links]{amsrefs}
\usepackage[capitalize, nameinlink]{cleveref}

\usepackage[utf8]{inputenc}
\usepackage[T1]{fontenc}
\usepackage{lmodern}
\usepackage[babel]{microtype}
\usepackage[english]{babel}

\usepackage{geometry}
\usepackage[presets={vec-cev}]{letterswitharrows}

\usepackage{enumitem}

\let\polishlcross=\l
\def\l{\ifmmode\ell\else\polishlcross\fi}

\let\emptyset=\varnothing

\let\theta=\vartheta
\let\rho=\varrho
\let\phi=\varphi

\def\NN{\mathbb N}

\newcommand{\Set}[1]{{\left\lbrace {#1} \right\rbrace}}

\def\set#1:#2{\Set{{#1} \colon {#2}}}

\theoremstyle{plain}
\newtheorem{thm}{Theorem}[section]

\newtheorem{prop}[thm]{Proposition}
\newtheorem{cor}[thm]{Corollary}

\newtheorem{obs}[thm]{Observation}

\newtheorem{claim}{Claim}

\theoremstyle{definition}

\def\lqedsymbol{\ifmmode$\lrcorner$\else{\unskip\nobreak\hfil
		\penalty50\hskip1em\null\nobreak\hfil$\rule{1.2ex}{1.2ex}$
		\parfillskip=0pt\finalhyphendemerits=0\endgraf}\fi}

\newenvironment{claimproof}[1][\proofname]
{%
	\proof[#1]%
}
{%
	\endproof%
}

\title{A star-comb lemma for finite digraphs}
\author{Florian Reich}
\address{Universit\"at Hamburg, Department of Mathematics, Bundesstrasse 55 (Geomatikum), 20146 Hamburg, Germany}
\email{florian.reich@uni-hamburg.de}


\begin{document}
	
\begin{abstract}
	It is well-known that for every set $U$ of vertices in a connected graph $G$ there is either a subdivided star in $G$ with a large number of leaves in $U$, or a comb in $G$ with a large number of teeth in $U$.
	
	In this paper we extend this property to directed graphs.
	More precisely, we prove that for every $n \in \NN$ and every sufficiently large set $U$ of vertices in a strongly connected directed graph~$D$,
	there exists a strongly connected butterfly minor of $D$ with $n$ teeth in $U$ that is either `shaped' by a star or `shaped' by a comb.
\end{abstract}
	
\maketitle
\section{Introduction}
One of the most fundamental results about unavoidable substructures is the existence of
\emph{stars}, that is graphs $K_{1,m}$ for $m \in \NN$, or paths in connected graphs~\cite{diestel}*{Proposition 9.4.1}:
\vspace{.5\baselineskip}
\begin{equation}
	\begin{aligned}\label{Fact:finite_undirected_simple}
		\parbox{\textwidth-3\parindent}{\emph{
				There exists a function $f$ such that, for every $n \in \NN$, every connected graph of order at least $f(n)$ contains either a star $K_{1,n}$ or a path $P_n$.
				}
				}
	\end{aligned}
	\tag{$\ast$}
\end{equation}

\vspace{.5\baselineskip}
\noindent
In the local context, the following variant of~\labelcref{Fact:finite_undirected_simple} is commonly used:
\vspace{.5\baselineskip}
\begin{equation}
	\begin{aligned}\label{Fact:finite_undirected}
		\parbox{\textwidth-3\parindent}{\emph{
				There exists a function $f'$ such that for every $n \in \NN$ and for every set $U$ of vertices in a connected graph $G$ with $|U| \geq f'(n)$ there is either a subdivided star or a comb in $G$ that has $n$ teeth in $U$.
		}}
	\end{aligned}
	\tag{$\ast \ast$}
\end{equation}

\vspace{.5\baselineskip}
\noindent
A \emph{comb} is the union of a path $P$ with disjoint (possibly trivial) paths having precisely their first vertex on $P$ and we refer to the last vertices of these paths as its \emph{teeth}.
We also refer to the leaves of a subdivided star as its \emph{teeth}.
Furthermore, we call a vertex $v$ of a comb a \emph{junction} if $v$ has degree $3$ or $v$ is a tooth and has degree $2$.

The unavoidable substructures of graphs of higher connectivity are also well-understood.
Cycles and bipartite graphs $K_{2,m}$ for $m \in \NN$ are the unavoidable topological minors of $2$-connected graphs~\cite{diestel}*{Proposition 9.4.2}.
Unavoidable minors for graphs of connectivity at least $3$ have been studied by Oporowski, Oxley, Thomas \cite{oporowski1993typical} and Geelen, Joeris \cites{geelen2016generalization,joeris2015connectivity}.
Furthermore, induced unavoidable structures have been investigated~\cites{chudnovsky2016unavoidable,allred2023unavoidable} with recent progress by Chudnovsky, Gartland, Hajebi, Lokshtanov and Spirkl~\cite{chudnovsky2024induced}*{Theorem 5.2} who proved a version of~\labelcref{Fact:finite_undirected} for induced graphs.
In all these results the size of the unavoidable substructure depends on the size of its host graph. More generally, one can relate to other parameters like for example tree-width, which has been studied in Robertson and Seymour's~\cite{robertson1986graph} grid theorem.

In the context of directed graphs, unavoidable substructures have been expressed in terms of butterfly minors.
Seymour and Thomassen~\cites{seymour1987characterization} discovered the notion of butterfly minor in 1987 and proved that odd bicycles are unavoidable butterfly minors of even directed graphs. Wiederrecht~\cite{wiederrecht2020digraphs} studied the unavoidable butterfly minors of strongly $2$-connected directed graphs.
Moreover, large unavoidable butterfly minors with respect to the directed tree-width~\cites{kawarabayashi2015directed} or the directed path-width~\cite{erde2020directed} of its host graph have been investigated by Kawarabayashi, Kreutzer and Erde.

The study of unavoidable substructures is tightly related to the question how certain classes of graphs can be generated from small families of base graphs.
Famous results are
Whitney's~\cite{whitney1931non} ear-decomposition for $2$-connected graphs and Tutte's~\cite{tutte1961theory} wheel theorem for $3$-connected graphs.
Whitney's ear decomposition applies also to strongly connected directed graphs.
For strongly $2$-connected directed graphs a decomposition was recently invented by Hatzel, Kreutzer, Protopapas, Stamoulis, Wiederrecht and the author~\cite{hatzel2024generating}.

Despite the fundamental nature of~\labelcref{Fact:finite_undirected_simple,Fact:finite_undirected}, so far no counterpart was known in the directed setting. We take on the task of investigating the unavoidable butterfly minors in strongly connected directed graphs.
Since for every pair $v, w$ of vertices there are two directed paths, namely a directed $v$--$w$~path and a directed $w$--$v$~path, the structure of strongly connected directed graphs is more complex than the structure of connected graphs.
This implies that the unavoidable butterfly minors of strongly connected directed graphs are more complex than stars and combs from \labelcref{Fact:finite_undirected}.

Given a graph $G$, let \emph{$\mathcal{D}(G)$} be the directed graph obtained from $G$ by replacing each edge $uv \in E(G)$ by directed edges $(u,v), (v,u)$.
We say that a directed graph $D$ is \emph{shaped by a star} (see~\cref{fig:shapedByStar}) if either $D$ is a directed cycle or there exists a subdivided star $S$ such that either
\begin{enumerate}[label=(\roman*)]
	\item\label{itm:star_1} $D= \mathcal{D}(S)$, or
	\item\label{itm:star_2} $D$ is obtained from $\mathcal{D}(S)$ by replacing the unique vertex $w \in V(\mathcal{D}(S))$ with $ d_S(w) > 2$ by a directed cycle of length $d_S(w)$ such that the strong components of $\mathcal{D}(S) - w$ become incident with distinct vertices of this directed cycle.
\end{enumerate}

\begin{figure}[ht]
	\begin{tikzpicture}
		
		\foreach \a in {1,2,4,6}{
			
			\draw[edge] ({1.9*cos(\a*360/6)},{1.9*sin(\a*360/6)}) to ({0.1*cos(\a*360/6)},{0.1*sin(\a*360/6)});
			\draw[edge]  ({0.1*cos(\a*360/6)},{0.1*sin(\a*360/6)}) to ({1.9*cos(\a*360/6)},{1.9*sin(\a*360/6)});
			\draw[LavenderMagenta] ({2*cos(\a*360/6)},{2*sin(\a*360/6)}) node {\huge .};
		}
		
		\foreach \a in {3,5}{
			\draw[edge] ({0.9*cos(\a*360/6)},{0.9*sin(\a*360/6)}) to ({0.1*cos(\a*360/6)},{0.1*sin(\a*360/6)});
			\draw[edge]  ({0.1*cos(\a*360/6)},{0.1*sin(\a*360/6)}) to ({0.9*cos(\a*360/6)},{0.9*sin(\a*360/6)});
			
			\draw[edge] ({1.9*cos(\a*360/6)},{1.9*sin(\a*360/6)}) to ({1.1*cos(\a*360/6)},{1.1*sin(\a*360/6)});
			\draw[edge]  ({1.1*cos(\a*360/6)},{1.1*sin(\a*360/6)}) to ({1.9*cos(\a*360/6)},{1.9*sin(\a*360/6)});
			\draw[] ({1*cos(\a*360/6)},{1*sin(\a*360/6)}) node {\huge .};
			\draw[LavenderMagenta] ({2*cos(\a*360/6)},{2*sin(\a*360/6)}) node {\huge .};
		}
		
		\draw (0,0) node {\huge.};
		
		\foreach \a in {2,3,4,6}{
			
			\draw[edge] ({5+1.9*cos(\a*360/6)},{1.9*sin(\a*360/6)}) to ({5+1.1*cos(\a*360/6)},{1.1*sin(\a*360/6)});
			\draw[edge]  ({5+1.1*cos(\a*360/6)},{1.1*sin(\a*360/6)}) to ({5+1.9*cos(\a*360/6)},{1.9*sin(\a*360/6)});
			\draw[edge]  ({5+1*cos(\a*360/6+5)},{1*sin(\a*360/6+5)}) to[] ({5+1*cos((\a+1)*360/6-5)},{1*sin((\a+1)*360/6-5)});
			
			\draw[]  ({5+1*cos(\a*360/6)},{1*sin(\a*360/6)}) node {\huge.};
			\draw[LavenderMagenta] ({5+ 2*cos(\a*360/6)},{2*sin(\a*360/6)}) node {\huge .};
		}
		
		\foreach \a in {1,5}{
			
			\draw[edge] ({5+1.4*cos(\a*360/6)},{1.4*sin(\a*360/6)}) to ({5+1.1*cos(\a*360/6)},{1.1*sin(\a*360/6)});
			\draw[edge]  ({5+1.1*cos(\a*360/6)},{1.1*sin(\a*360/6)}) to ({5+1.4*cos(\a*360/6)},{1.4*sin(\a*360/6)});
			\draw[edge] ({5+1.9*cos(\a*360/6)},{1.9*sin(\a*360/6)}) to ({5+1.6*cos(\a*360/6)},{1.6*sin(\a*360/6)});
			\draw[edge]  ({5+1.6*cos(\a*360/6)},{1.6*sin(\a*360/6)}) to ({5+1.9*cos(\a*360/6)},{1.9*sin(\a*360/6)});
			\draw[edge]  ({5+1*cos(\a*360/6+5)},{1*sin(\a*360/6+5)}) to[] ({5+1*cos((\a+1)*360/6-5)},{1*sin((\a+1)*360/6-5)});
			
			\draw[]  ({5+1.5*cos(\a*360/6)},{1.5*sin(\a*360/6)}) node {\huge.};
			\draw[]  ({5+1*cos(\a*360/6)},{1*sin(\a*360/6)}) node {\huge.};
			\draw[LavenderMagenta] ({5+ 2*cos(\a*360/6)},{2*sin(\a*360/6)}) node {\huge .};
		}
		
		\foreach \a in {1,2,...,6}{
			\draw[edge]  ({10+2*cos(\a*360/6+2)},{2*sin(\a*360/6+2)}) to[] ({10+2*cos((\a+1)*360/6-2)},{2*sin((\a+1)*360/6-2)});
			\draw[LavenderMagenta] ({10+ 2*cos(\a*360/6)},{2*sin(\a*360/6)}) node {\huge .};
		}
		
	\end{tikzpicture}
	\caption{Three directed graphs shaped by a star.}
	\label{fig:shapedByStar}
\end{figure}

\noindent
A directed graph $D$ is \emph{shaped by a comb} (see~\cref{fig:shapedByComb}) if there exists a comb $C$ such that either
\begin{enumerate}[label=(\alph*)]
	\item\label{itm:comb_1} $D= \mathcal{D}(C)$, or
	\item\label{itm:comb_2} $D$ is obtained from $\mathcal{D}(C)$ by replacing each junction $j$ of $C$ by a directed cycle of length $3$ such that, if $d_C(j) = 3$, the three strong components of $\mathcal{D}(C) - j$ become incident with distinct vertices of this directed cycle and otherwise the tooth $j$ is a vertex of this directed cycle and the two strong components of $\mathcal{D}(C) - j$ become incident with distinct other vertices of this directed cycle.
\end{enumerate}

\begin{figure}[ht]
	\begin{tikzpicture}        
		\foreach \a in {-6, -4, -2}{
			\draw[LavenderMagenta] (2*\a,1.4) node {\huge.};
			\draw[] (2*\a,0) node {.};   
			
			\draw[edge] ({2*\a+0.1},0) to ({2*\a+1.9},0);
			\draw[edge] ({2*\a+1.9},0) to ({2*\a+0.1},0);
			\draw[edge] ({2*\a},0.07) to ({2*\a},1.33);
			\draw[edge] ({2*\a},1.33) to ({2*\a},0.07);
		}
		
		\draw[] (-10,0) node {.};
		
		\draw[edge] (-9.9,0) to (-8.1,0);
		\draw[edge] (-8.1,0) to (-9.9,0);
		\draw[edge] (-10,0.07) to (-10,0.63);
		\draw[edge] (-10,0.77) to (-10,1.33);
		\draw[edge] (-10,1.33) to (-10,0.77);
		\draw[edge] (-10,0.63) to (-10,0.07);
		\draw[] (-10,1.4) node {.};
		\draw[] (-10,0.7) node {.};
		
		\draw[LavenderMagenta] (-10,1.4) node {\huge.};
		
		\draw[edge] (-5.9,0) to (-4.1,0);
		\draw[edge] (-4.1,0) to (-5.9,0);
		\draw[LavenderMagenta] (-6,0) node {\huge.};
		
		\draw[edge] (-2,0.07) to (-2,1.33);
		\draw[edge] (-2,1.33) to (-2,0.07);
		
		\draw[LavenderMagenta] (-2,1.4) node {\huge.};
		\draw[] (-2,0) node {.};

	\end{tikzpicture}

	\begin{tikzpicture}
		
		\draw[edge][] (0.1,0) to (1.5,0);
		\draw[edge] ({1.5},0) to ({0.1},0);
		\draw[edge] ({0},0.07) to ({0},1.33);
		\draw[edge] ({0},1.33) to ({0},0.07);
		
		\draw[edge] ({10},0.07) to (10,1.33);
		\draw[edge] (10,1.33) to (10,0.07);
		\draw[] ({0},0) node {.};
		\draw[] ({10},0) node {.};
		
		\foreach \a in {1, ..., 4}{
			\draw[edge] ({2*\a-0.3},0) to ({2*\a+0.3},0);
			\draw[edge] ({2*\a+0.3},0.07) to ({2*\a+0.1},0.49);
			\draw[edge] ({2*\a-0.1},0.49) to ({2*\a-0.3},0.07);
			
			\draw[] ({2*\a-0.4},0) node {.};
			\draw[] ({2*\a+0.4},0) node {.};
			\draw[] ({2*\a},0.56) node {.};
			
			\draw[edge] ({2*\a+1.5},0) to ({2*\a+0.5},0);
		}
		
		\foreach \a in {2, 4}{
			\draw[LavenderMagenta] (2*\a,1.4) node {\huge.};
			
			\draw[edge] ({2*\a},0.63) to ({2*\a},1.33);
			\draw[edge] ({2*\a},1.33) to ({2*\a},0.63);
		}
		
		\foreach \a in {1,2,3}{
			\draw[edge] ({2*\a+0.5},0) to ({2*\a+1.5},0);
		}
		
		\draw[LavenderMagenta] (6,0.56) node {\huge.};
		
		\draw[LavenderMagenta] (2,1.4) node {\huge.};
		\draw[] (2,0.98) node {.};
		\draw[edge] (2,0.63) to (2,0.91);
		\draw[edge] (2,0.91) to (2,0.63);
		\draw[edge] (2,1.05) to (2,1.33);
		\draw[edge] (2,1.33) to (2,1.05);
		
		\draw[LavenderMagenta] (0,1.4) node {\huge.};
		\draw[LavenderMagenta] (10,1.4) node {\huge.};
		
		\draw[edge] (9.2,0) to (9.9,0);
	\end{tikzpicture}
	\caption{Two directed graphs shaped by a comb.}
	\label{fig:shapedByComb}
\end{figure}

\noindent
Let $T$ be a subdivided star or a comb. Then the \emph{teeth} of a directed graph $D$ shaped by $T$ are the vertices of $D$ corresponding to the teeth of $T$.
Furthermore, we call all vertices of a directed cycle \emph{teeth}.

We remark that all directed graphs shaped by a star or shaped by a comb are strongly connected.
Our main result reads as follows:
\begin{restatable}{thm}{finite}\label{finite}
	There exists a function $g$ such that for every $n \in \NN$ and for every set $U$ of vertices of strongly connected directed graph $D$ with $|U| \geq g(n)$ there is a butterfly minor of $D$ that is either shaped by a star or shaped by a comb, and has $n$ teeth in $U$.
\end{restatable}
\noindent
We derive a variant of~\labelcref{Fact:finite_undirected_simple} for directed graphs from \cref{finite}:
\begin{cor}
	There exists a function $g$ such that, for every $n \in \NN$, every strongly connected directed graph $D$ of order at least $g(n)$ contains $\mathcal{D}(K_{1,n}), \mathcal{D}(P_n)$ or a directed cycle of length $n$ as a butterfly minor.
\end{cor}
All five types of graphs in~\cref{finite} are indeed necessary for a statement like~\cref{finite}:
Let $D$ be an arbitrary directed graph of one of these types.
It suffices to show that no essential butterfly minor of $D$ is of a different type than $D$.
A butterfly minor is \emph{essential} if it is strongly connected and preserves a large number of teeth of its host graph.
\begin{itemize}
	\item If $D$ is of type~\labelcref{itm:star_1}, every essential butterfly minor of $D$ contains a vertex of high in-degree since the centre vertex has an in-edge for each tooth of $D$ that it preserves.
	\item If $D$ is of type~\labelcref{itm:star_2}, every essential butterfly minor of $D$ contains a long directed cycle whose vertices are not teeth of $D$ as the centre cycle has a vertex for each tooth of $D$ that is preserved.
	\item If $D$ is of type~\labelcref{itm:comb_1}, the essential butterfly minors of $D$ contain a copy of $\mathcal{D}(P)$ for some long path $P$ with the property that many vertices of $\mathcal{D}(P)$ have in-degree $3$ or are teeth of $D$. This is due to the fact that remaining teeth of $D$ must be connected by a copy of~$\mathcal{D}(P)$.
	\item If $D$ is of type~\labelcref{itm:comb_2}, every essential butterfly minor of $D$ holds a directed cycle of length $3$ for each tooth (but the two outer) it preserves. Thus every essential butterfly minor of $D$ contains many directed cycles of length $3$.
	\item Essential butterfly minors of directed cycles are again directed cycles that contain many teeth of $D$.
\end{itemize}
These properties ensure that no essential butterfly minor of $D$ is of a different type than $D$.

We remark that if we consider strong minors rather than butterfly minors in the statement of~\cref{finite}, types~\labelcref{itm:star_2} and~\labelcref{itm:comb_2} can be omitted since such graphs can be strongly contracted to graphs of type~\labelcref{itm:star_1} or~\labelcref{itm:comb_1}, respectively.

In the second paper of this series~\cite{infinite} we extend a variant of~\labelcref{Fact:finite_undirected} for infinite graphs, which is known as the star-comb lemma and is a standard tool in infinite graph theory, to directed graphs.

This paper is organised as follows.
We introduce basic notations and butterfly minors in~\cref{sec:preliminaries}.
In~\cref{sec:double_rays_stars} we investigate the strong connectivity of pairs of vertices: We show that every pair of directed paths (e.g. a directed $a$--$b$~path and a directed $b$--$a$~path for vertices $a$ and $b$) can be simplified to a pair of \emph{laced} paths.
In~\cref{sec:centre} we show that for every large set $U$ of vertices there exists a strongly connected butterfly minor of a certain type containing many vertices of $U$:
These butterfly minors consist of a strongly connected `centre' together with a family of subgraphs each ensuring that some element of $U$ is strongly connected to the centre.
Finally, we prove~\cref{finite} in~\cref{sec:proof} by analysing the structure of this family of subgraphs.

\section{Preliminaries}\label{sec:preliminaries}
For standard notations we refer to Diestel's book~\cite{diestel}.
We set $[\ell]:= \{1, \dots, \ell\}$ for $\ell \in \NN$.
A \emph{directed $x_1$--$x_\ell$~path}, or simply a \emph{directed path}, of a directed graph is an alternating sequence $P=x_1e_2x_2 \dots x_{\ell -1} e_{\ell - 1} x_\ell$ of distinct vertices $(x_i)_{i \in [\ell]}$ and distinct edges $(e_i)_{i \in [\ell - 1]}$ such that $x_i$ is the tail of $e_i$ and $x_{i+1}$ is the head of $e_i$.
We call the vertices $x_2, \dots, x_{\ell - 1}$ the \emph{internal vertices} of $P$.
Further, we call $x_1$ the \emph{startvertex} and $x_\ell$ the \emph{endvertex} of $P$.
We also refer to the directed graph induced by $P$ as a \emph{directed path}.
Given two sets of vertices $X, Y$, the directed path $P$ is a \emph{directed $X$--$Y$~path} if $x_1 \in X$, $x_\ell \in Y$ and the internal vertices of $P$ are not contained in $X \cup Y$.
Further, for subgraphs $A, B$, we refer to directed $V(A)$--$V(B)$~paths as \emph{directed $A$--$B$~paths}.

Given some directed path $P$ and a vertex $x \in V(P)$ we define $Px$ to be the initial segment of $P$ ending in $x$ and define $xP$ analogously.
Given further some directed path $Q$ with $x \in V(P) \cap V(Q)$, the term $P x Q$ refers to the directed path obtained by the concatenation of $Px$ and $xQ$ along $x$.
We define an order on the vertex set of a directed path $P$ as follows: for $v, w \in V(P)$, we say $v \leq_P w$ if $w\in V(vP)$.

A \emph{double path} is the directed graph $\mathcal{D}(P)$ obtained from an undirected path $P$.
A \emph{directed cycle} is a sequence $C = x_1e_2x_2 \dots x_{\ell -1} e_{\ell - 1} x_\ell$ of distinct vertices $(x_i)_{i \in [\ell - 1]}$ with $x_1 = x_\ell$ and distinct edges $(e_i)_{i \in [\ell - 1]}$ such that $x_i$ is the tail of $e_i$ and $x_{i+1}$ is the head of $e_i$.
We call a directed graph an \emph{in-arborescence}, if its underlying undirected graph is a rooted tree and all edges are directed towards the root.
The terms root and leaves transfer from the underlying undirected graph to the in-arborescence.
\emph{Out-arborescences} are defined analogously.

A directed graph $D$ is called \emph{strongly connected}, if for every $u, v \in V(D)$ there exists a directed $u$--$v$~path in $D$.

Throughout this paper we implicitly make use of Ramsey's theorem:
\begin{thm}{\cite{diestel}*{Theorem 9.1.1}} \label{ramsey1}
	There exists a function $\phi$ such that, for every $n \in \NN$ and every $k \in \NN$, every complete undirected graph on $\phi(k,n)$ vertices with a $k$-colouring of its edges contains a monochromatic clique of size $n$.
\end{thm}

\subsection{Butterfly minors}

	An edge $e=(u,v)\in E(D)$ of a directed graph $D$ is called \emph{butterfly contractible} if $e$ is the only outgoing edge of $u$ or the only incoming edge of $v$.
	A directed graph $H$ is a \emph{butterfly minor} of a directed graph $D$ if $H$ can be obtained from a subgraph of $D$ by repeatedly contracting butterfly contractible edges.
	
\begin{obs}
	Let $D_1$, $D_2$, $D_3$ be directed graphs such that $D_1$ is a butterfly minor of $D_2$ and $D_2$ is a butterfly minor of $D_3$. Then $D_1$ is butterfly-minor of $D_3$.
\end{obs}

The vertex set of a butterfly minor is in general not a subset of the vertex set of its host graph.
Therefore in the statement of~\cref{finite} it is not immediately clear how a butterfly minor can have its teeth in $U$.
To overcome this problem, we identify (in each contraction step) the vertex $x_e$ obtained by butterfly contracting an edge $e = (u,v)$ with $u$ or $v$: we identify $x_e$ with $u$ if $e$ is the only in-going edge of $v$ and we identify $x_e$ with $v$ if $e$ is the only out-going edge of $u$.
Note that we allow $x_e$ to be identified with both $u$ and $v$.
In this way some vertex of a butterfly minor is an element of $U$ if we identified it with a vertex of $U$.

We show that these identifications do not create new connectivity:

\begin{prop}\label{propconnectedness}
	Let $H$ be a butterfly minor of a directed graph $D$ and $v,w \in V(H)$. If there exists a directed  $v$--$w$~path in $H$, then there exists a directed $v$--$w$~path in $D$.
\end{prop}

	\begin{proof}
		Since $H$ is obtained from $D$ by a sequence of butterfly contractions and edge deletions we can assume that $H$ is obtained from $D$ by butterfly contracting a single edge $e = (u,u')$.
		Let $x_e$ be the vertex of $H$ obtained by butterfly contracting $e$.
		
		By the definition of butterfly contraction, every directed path $P$ in $H$ extends to a directed path $P^*$ in $D$ under adding the edge $e$ whenever possible.
		If $e$ is the only ingoing edge of $u'$ and $P$ contains $x_e$, then $P^*$ contains $u$.
		Similarly, if $e$ is the only outgoing edge of $u$ and $P$ contains $x_e$, then $P^*$ contains $u'$.
		The way we identify vertices ensures that directed $v$--$w$~paths in $H$ turn into directed $v$--$w$~paths in $D$.
	\end{proof}
	
	Finally, we present two classes of substructures that can be butterfly contracted to single vertices:
	We call a rooted directed graph $T_{\textrm{in}}$ an \emph{in-arborescence} if the underlying undirected graph of $T_{\textrm{in}}$ is a tree and all its edges are oriented towards its root.
	An in-arborescence $T_{\textrm{in}}$ in a directed graph $D$ can be butterfly contracted to a single vertex if all vertices of $T_{\textrm{in}}$ but its root have out-degree one in $D$. In particular, the vertex obtained by butterfly contracting $T_{\textrm{in}}$ is identified with the root of $T_{\textrm{in}}$.
	We define out-arborescences analogously and note that an out-arborescence in a directed graph $D$ can be butterfly contracted to a single vertex if all its vertices but its root have in-degree one in $D$.

\section{Laced paths} \label{sec:double_rays_stars}
In this section we investigate the interaction of pairs of directed paths. We show that every two directed paths can be simplified to a pair of \emph{laced paths}. This provides a simple witness for the strong connectivity of two given vertices.

Two directed paths $P$ and $Q$ are \emph{laced} if either the paths $P$ and $Q$ are disjoint or there exist
$x_1, y_1, \dots, x_\ell, y_\ell \in V(P) \cap V(Q)$
for some $\ell \in \NN$ such that
\begin{itemize}
	\item $	 x_1 \leq_P y_1 <_P \dots <_P x_{\ell} \leq_P y_{\ell}$,
	\item $x_\ell \leq_Q y_\ell <_Q \dots <_Q x_1 \leq_Q y_1$,
	\item $x_i P y_i = x_i Q y_i$ for every $i \in [\ell]$,
	\item all segments $y_\ell Q x_{\ell - 1}, \dots, y_2 Q x_{ 1}$ are internally disjoint to $P$, and
	\item $Q x_\ell$ intersects $P$ only in $x_\ell$ and $y_1 Q$ intersects $P$ only in $y_1$.
\end{itemize}
See \cref{fig:laced_paths} for an illustration.

\begin{figure}[ht]
	\begin{tikzpicture}

		\foreach \a in {1,...,9}{
			\draw[edge][PastelOrange] ({\a+0.1},1) to ({\a+0.9},1);
			\draw[] (\a,1) node {\huge .};
			\draw[] ({\a+1},1) node {\huge .};
		}
		\draw[edge, dashed][CornflowerBlue] (3.1,1) to (3.9,1);
		\draw[edge, dashed][CornflowerBlue] (5.1,1) to (5.9,1);
		\draw[edge, dashed][CornflowerBlue] (7.1,1) to (7.9,1);
		
		\draw[edge][CornflowerBlue] (4,0.9) to[bend left] (2.1,0);
		\draw[edge][CornflowerBlue] (8,0.9) to[bend left] (5,0.9);
		\draw[edge][CornflowerBlue] (6,1.1) to[bend right] (3,1.1);
		\draw[edge][CornflowerBlue] (8.9,2) to[bend right] (7,1.1);
		
		\draw[] (2,0) node {\huge .};
		\draw[] (9,2) node {\huge .};

		\draw (3,1.6) node {$x_1$};
		\draw (4,0.4) node {$y_1$};
		\draw (5,0.4) node {$x_2$};
		\draw (6,1.6) node {$y_2$};
		\draw (7,1.6) node {$x_3$};
		\draw (8,0.4) node {$y_3$};

		\draw[CornflowerBlue] (4.5,2) node {$Q$};
		
		\draw[PastelOrange] (2,2) node {$P$};

	\end{tikzpicture}
	\caption{Two laced directed paths.}
	\label{fig:laced_paths}
\end{figure}

\begin{prop}\label{prop:laced_paths} \cite{hatzel2024generating}*{Lemma 4.3}
	Let $D$ be a directed graph and let $a,b \in V(D)$.
	Further, let $P$ be some directed path and let $Q$ be a directed $a$--$b$~path in $D$.
	Then there exists a directed $a$--$b$~path $Q'$ with $Q' \subseteq P \cup Q$ such that $P$ and $Q'$ are laced.
\end{prop}
To keep the paper self-contained, we present a short proof of \cref{prop:laced_paths} that follows the lines of the proof given in~\cite{hatzel2024generating}:
\begin{proof}
	Let $Q' \subseteq Q \cup P$ be a directed $a$--$b$~path that minimises the number of edges outside $E(P)$.
	We show that $P$ and $Q'$ are laced. 
	
	If $Q'$ and $P$ are disjoint, then they are laced by definition.
	Otherwise, let $P_1, \dots, P_{\ell}$ be the weakly connected components of $P \cap Q'$ in the order in which they appear on $P$. Note that every such component is a directed path, possibly trivial.
	For every $i \in [\ell]$ let $x_{i}$ be the startvertex of $P_i$ and let $y_i$ be the endvertex of $P_i$.
	Then $x_1 \leq_P y_1 <_P \dots <_P x_\ell \leq_P y_\ell$.
	
	Now, if $Q'$ and $P$ are not laced, then there must be an index $i\in[\ell-1]$ such that $y_i$ appears on $Q'$ before $x_{i+1}$.
	But $P' :=  x_i P y_{i+1}$ is a directed subpath of $P$ that is internally disjoint from $Q'$ and thus the concatenation $Q'':= Q' y_i  P' x_{i+1} Q'$ is a directed $a$--$b$~path such that $Q''$ has fewer edges outside $E(P)$, a contradiction to the choice of $Q'$.
\end{proof}

\begin{obs}\label{obs:double_path}
	Let $D$ be a directed graph and let $a,b \in V(D)$.
	Further, let $P$ be a directed $a$--$b$~path and let $Q$ be a directed $b$--$a$~path such that $P$ and $Q$ are laced.
	Then $P \cup Q$ can be butterfly contracted to a double path starting in $a$ and ending in $b$.
\end{obs}

We remark that in the setting of \cref{obs:double_path} there exist directed cycles $(C_j)_{j \in [p]}$ with the property that for every $j \neq k \in [p]$ $C_j\cap C_k$ is a directed path if $|j-k|=1$ and empty otherwise, and $\bigcup_{i \in [p]} C_i = P \cup Q$.

\section{Constructing the centre} \label{sec:centre}
In this section we establish the foundation for the proof of~\cref{finite}.
Given a set of vertices $U$ in a strongly connected directed graph, we prove the existence of a strongly connected butterfly minor consisting of a family of disjoint double paths that have an endpoint in $U$ and are `attached' to a strongly connected directed graph of certain type:

\begin{restatable}{lem}{propgroundwork}\label{propgroundwork}
	There exists a function $\alpha$ such that for every $n \in \NN$ and every set $U$ of vertices of a directed graph $D$ with $|U| \geq \alpha(n)$ there is a butterfly minor that consists either
	\begin{enumerate}[label=(\roman*)]
		\item\label{itm:centre_4} 
		of a directed cycle containing $n$ elements of $U$ or of a family of directed cycles $(C_j)_{j \in [p]}$ with the property that for every $j \neq k$ $C_j\cap C_k$ is a directed path if $|j-k|=1$ and there is an $n$-element subset $U' \subseteq \bigcup_{j \in [p]} V(C_j) \cap U$ such that each directed cycle of $(C_j)_{j \in [p]}$ contains at most one element of $U'$, or
		\item of a family of disjoint double paths $(F_i: i \in [n])$ where $F_i$ has endpoints $u_i \in U$ and $z_i$, edges $((a_i, z_i): i \in [n])$, $((z_i, b_i): i \in [n])$ and a directed graph $A$ with $A \cap \bigcup_{i \in [n]} F_i = \emptyset$ such that
		\begin{enumerate}[label=(ii--\alph*)]
			\item\label{itm:centre_1} $A$ is a single vertex with $A=\{a_i\}=\{b_i\}$ for every $i \in [n]$,
			\item\label{itm:centre_2} $A$ is a directed cycle containing $a_i, b_i$ for every $i \in [n]$ and the elements of $(a_i: i \in [n])$ are distinct, or
			\item\label{itm:centre_3} $A$ consists of a sequence of directed cycles $(C_j)_{j \in [p]}$ such that for every $j \neq k$ $C_j\cap C_k$ is a directed path if $|j-k|=1$, and empty otherwise, $a_i, b_i \in V(\bigcup_{j \in [p]} C_j)$ for every $i \in [n]$ and each directed cycle $(C_j)_{j \in [p]}$ contains at most one element of $(a_i: i \in [n])$.
		\end{enumerate}
	\end{enumerate}
\end{restatable}
\noindent
For proving~\cref{finite} we will apply \cref{propgroundwork}, and if the butterfly minor is of types~\labelcref{itm:centre_2} or \labelcref{itm:centre_3}, we analyse how the vertices $(a_i: i \in [n])$ and $(b_i: i \in [n])$ are distributed in the directed cycle or the union of directed cycles.

As a first step towards the proof of~\cref{propgroundwork} we show the existence of one of the following subgraphs:

\begin{prop}\label{propmainstructure}
	There exists a function $\beta$ such that for every $n \in \NN$ and every set $U$ of vertices of a directed graph $D$ with $|U| \geq \beta(n)$ there is a strongly-connected subgraph $A$ of $D$ such that:
	\begin{enumerate}[label=(\alph*)]
		\item\label{itm:main_structure_1} $A$ is a single vertex and there is a family $(P_i: i \in [n])$ of non-trivial directed $A$--$U$~paths that intersect only in $A$, or
		\item\label{itm:main_structure_2} $A$ is a directed cycle and there is a family $(P_i: i \in [n])$ of (possibly trivial) disjoint directed $A$--$U$~paths, or
		\item\label{itm:main_structure_3} $A$ is the union of a sequence of directed cycles $(C_j)_{j \in [p]}$ with the property that for every $j \neq k$ $C_j\cap C_k$ is a directed path if $|j-k|=1$ and empty otherwise; there is a family $(P_i: i \in [n])$ of (possibly trivial) disjoint directed $A$--$U$~paths such that every directed cycle in $(C_j)_{j \in [p]}$ contains at most one startvertex of the directed paths $(P_i: i \in [n])$.
	\end{enumerate}
\end{prop}

See \cref{fig:type} for an illustration of such subgraphs.

\begin{figure}[ht]
	\begin{tikzpicture}
		
		\foreach \a in {1,2,...,4}{
			
			\draw[edge]  ({0.1*cos(\a*360/4)},{0.1*sin(\a*360/4)}) to ({1.9*cos(\a*360/4)},{1.9*sin(\a*360/4)});
			\draw[LavenderMagenta] ({2*cos(\a*360/4)},{2*sin(\a*360/4)}) node {\huge .};
			\draw ({2.3*cos((\a+1)*360/4)},{2.3*sin((\a+1)*360/4)}) node {$u_\a$};
		}
		\draw[CornflowerBlue] (0,0) node {\huge.};
		
		\foreach \a in {1,2,...,4}{
			
			\draw[edge]  ({6+1.1*cos(\a*360/4)},{1.1*sin(\a*360/4)}) to ({6+1.9*cos(\a*360/4)},{1.9*sin(\a*360/4)});
			\draw[edge][CornflowerBlue]  ({6+1*cos(\a*360/4+5)},{1*sin(\a*360/4+5)}) to[] ({6+1*cos((\a+1)*360/4-5)},{1*sin((\a+1)*360/4-5)});
			
			\draw ({6+2.3*cos((\a+1)*360/4)},{2.3*sin((\a+1)*360/4)}) node {$u_\a$};
			
			\draw[CornflowerBlue]  ({6+1*cos(\a*360/4)},{1*sin(\a*360/4)}) node {\huge.};
			\draw[LavenderMagenta] ({6+ 2*cos(\a*360/4)},{2*sin(\a*360/4)}) node {\huge .};
		}
		\draw[CornflowerBlue] (6,0) node {$A$};
		\draw[CornflowerBlue] (0.5,0.5) node {$A$};

	\end{tikzpicture}
	\begin{tikzpicture}
		
		\foreach \a in {1,5, 9}{
			\draw[edge][CornflowerBlue] ({\a+0.1},0) to ({\a+1.9},0);
			\draw[edge][PastelOrange] ({\a+2},0.1) to ({\a+2},1.9);
			\draw[edge][CornflowerBlue, dashed] ({\a+2},0.1) to ({\a+2},1.9);	
			\draw[edge][CornflowerBlue] ({\a+1.9},2) to ({\a+0.1},2);
			\draw[edge][CornflowerBlue] (\a,1.9) to ({\a},0.1);
			\draw[edge][PastelOrange, dashed] ({\a},1.9) to ({\a},0.1);	
			
			\draw[edge][PastelOrange] ({\a+3.9},0) to ({\a+2.1},0);
			\draw[edge][PastelOrange] ({\a+2.1},2) to ({\a+3.9},2);
		}
		
		\draw[edge][CornflowerBlue] (1,1.9) to (1,0.1);
		\draw[edge][PastelOrange] (13,1.9) to ({13},0.1);
		
		\draw[edge][] (1,2.1) to ({1},2.9);
		\draw[] ({1},2) node {\huge.};
		
		\draw[white, fill=white] (4,2) circle (0.1);
		\draw[white, fill=white] (7,1) circle (0.1);
		
		\draw[edge][] (4,2.1) to ({4},2.9);
		\draw[] ({4},2) node {\huge.};
		\draw[edge][] (6.9,1) to ({6.1},1);
		\draw[] ({7},1) node {\huge.};
		\draw[edge][] (11.1,0.1) to ({11.9},0.9);
		\draw[] ({11},0) node {\huge.};
		
		\draw[edge][PastelOrange] (3,2) to (3.9,2);
		\draw[edge][CornflowerBlue] (7,0.8) to (7,0.9);
		
		\draw[LavenderMagenta] ({1},3) node {\huge.};
		\draw ({1.3},3) node {$u_1$};
		\draw[LavenderMagenta]({4},3) node {\huge.};
		\draw({4.3},3) node {$u_2$};
		\draw[LavenderMagenta] ({6},1) node {\huge.};
		\draw ({5.7},1) node {$u_3$};
		\draw[LavenderMagenta] ({12},1) node {\huge.};
		\draw ({12.3},1) node {$u_4$};
		
		\foreach \a in {1, 3, 5}{
			\draw[CornflowerBlue] ({2*\a},0.5) node {$C_\a$};
			\draw[] ({2*\a-1},0) node {\huge.};
			\draw[] ({2*\a-1},2) node {\huge.};
		}
		
		\foreach \a in {2, 4, 6}{
			\draw[PastelOrange] ({2*\a},0.5) node {$C_\a$};
			\draw[] ({2*\a-1},0) node {\huge.};
			\draw[] ({2*\a-1},2) node {\huge.};
		}
		
		\draw[] (13,0) node {\huge.};
		\draw[] (13,2) node {\huge.};

	\end{tikzpicture}
	\caption{Directed graphs of types~\labelcref{itm:main_structure_1,itm:main_structure_2,itm:main_structure_3} in~\cref{propmainstructure}.}	
	\label{fig:type}
\end{figure}
\begin{proof}
	Let $n \in \NN$ be arbitrary, and fix $\beta(n), m \in \NN$ such that $\beta(n) \gg m \gg n$.
	Further, let $U$ be an arbitrary set of vertices of a strongly connected directed graph $D$ with $|U| \geq \beta(n)$.
	Let $T$ be an out-arborescences in $D$ that contains all vertices of $U$ and whose leaves are contained in $U$.
	Such an out-arborescence can be constructed by depth-first search since $D$ is strongly-connected.
	
	If $T$ contains a vertex $x$ of out-degree at least $n$, set $A:=\{x\}$ and let $(P_i: i \in [n])$ be a family of maximal directed paths in $T$ starting in $x$, which witness a subgraph of type~\labelcref{itm:main_structure_1}.
	
	Thus we can assume that $T$ contains only vertices of out-degree at most $n - 1$.
	Applying \labelcref{Fact:finite_undirected} to the underlying undirected graph of $T$ provides a comb with $4m$ teeth in $U$.
	Note that the back of this comb splits into two directed paths.
	This implies that there exists a directed path $R$ in $T$ that contains either $m$ elements of $U$ or $m$ vertices of out-degree at least $2$ in $T$.
	These $m$ elements of $U$ or the $m$ vertices of out-degree at least $2$ witness that there exists a family $(\hat{P}_i: i \in [m])$ of (possibly trivial) disjoint directed $R$--$U$~paths in $T$.
	
	Let $S$ be a directed path starting in the endvertex of $R$ and ending in the startvertex of $R$ that is laced with $R$, which is possible since $D$ is strongly connected and by \cref{prop:laced_paths}.
	Then $R \cup S= \bigcup_{i \in [p]} C_i$ for a sequence $(C_i)_{i \in [p]}$ of directed cycles with the property that for every $j \neq k$ $C_j\cap C_k$ is a directed path if $|j-k|=1$ and empty otherwise.
	Each directed path $\hat{P}_i$ contains a terminal segment $P_i$ that is a directed $( \bigcup_{i \in [p]} C_i)$--$U$~path.
	
	If there exists $j \in [p]$ such that at least $n$ elements of $(P_i: i \in [m])$ start in the directed cycle $C_j$, then the union of $C_j$ and these $n$ elements of $(P_i: i \in [m])$ is the desired subgraph of type~\labelcref{itm:main_structure_2}.
	Otherwise, by the choice of $m$, there exist $3n$ directed cycles in $(C_j: j \in [p])$ in which a directed path of $(P_i: i \in [m])$ starts. Therefore, we can pick a family $\mathcal{C}$ of $n$ directed cycles in $(C_j: j \in [p])$ in which a directed path of $(P_i: i \in [m])$ starts such that the indices of the directed cycles in $\mathcal{C}$ differ pairwise by at least $3$. Then the union of $\bigcup_{j \in [p]} C_j$ and one directed path of $(P_i: i \in [m])$ starting in $C$ for each element $C \in \mathcal{C}$ is a subgraph of type~\labelcref{itm:main_structure_3}.
\end{proof}

\begin{prop}\label{lem:minimal_property}
	Let $U$ be a set of vertices of a directed graph $D$ such that $D$ is edge-minimal under the property of being strongly connected and containing $U$.
	Further, let $A$ be a strongly-connected subgraph of $D$.
	Then for every edge $e$ with head in $V(D) \setminus V(A)$ and tail in $V(A)$ there exists $v \in U$ for which there is no directed $A$--$v$~path in $D - e$.
\end{prop}
\begin{proof}
	By edge-minimality of $D$, there exist $u \neq v \in U$ such that every directed $u$--$v$~path in $D$ contains $e$.
	In particular, there exists a directed $u$--$A$~path $Q$ in $D-e$.
	
	Suppose for a contradiction that there exist a directed $A$--$v$~path $P$ in $D - e$.
	Since $A$ is strongly-connected there exists a directed $u$--$v$~path in $Q \cup A \cup P \subseteq D - e$, contradicting the choice of $u$ and $v$.
\end{proof}

\begin{proof}[Proof of \cref{propgroundwork}]
	Let $n \in \NN$ be arbitrary.
	We set $\alpha(n):= \beta(2n)$, where $\beta$ is the function from~\cref{propmainstructure}.
	Let $U$ be an arbitrary set of vertices of a strongly connected directed graph $D$ with $|U| \geq \alpha(n)$.
	We assume that $D$ is edge-minimal with the property that $U$ is strongly-connected and
	apply \cref{propmainstructure} to $U$ to obtain a strongly connected subgraph $A$ and a family $(P_i: i \in [2n])$ of directed paths as described in~\cref{propmainstructure}.
	
	If $n$ directed paths of $(P_i: i \in [2n])$ are trivial, then $A$ is of type~\labelcref{itm:main_structure_2} or~\labelcref{itm:main_structure_3}, i.e.\ $A$ is either a directed cycle or a union of directed cycles $(C_j)_{j \in [p]}$ with the property that for every $j \neq k$ $C_j\cap C_k$ is a directed path if $|j-k|=1$ and empty otherwise, and thus $A$ is the desired butterfly minor of type~\labelcref{itm:centre_4}.
	Otherwise, up to relabelling, we can assume that the directed paths $(P_i: i \in [n])$ are non-trivial.
	Let $e_i$ be the first edge of $P_i$ and let $\Tilde{P}_i$ be the terminal segment of $P_i$ that starts in the second vertex of $P_i$ for every $i \in [n]$.
	
	\begin{claim}\label{clm:notAvoidingFirstEdge}
		We can assume that every directed path ${P}_i$ for $i \in [n]$ has the property that there is no directed $A$--$\Tilde{P}_i$~path in $D-e_i$.
	\end{claim}
	\begin{claimproof}
		We assume that the number of directed paths in $(P_i: i \in [n])$ without this property is minimal under all possible choices for such directed paths $(P_i)_{i \in [n]}$.
		Suppose for a contradiction that there exists some $k \in [n]$ such that there is a directed $A$--$\Tilde{P}_k$~path in $D - e_k$.
		
		We apply \cref{lem:minimal_property} to the edge $e_k$ to obtain a vertex $v \in U$ such that there is no directed $A$--$v$~path in $D -e_k$.
		Let $P_k'$ be a directed $A$--$v$~path in $D$.
		Note that $e_k$ is the first edge of $P_k'$ by choice of $e_k$.
		Furthermore, every $P_j$ for $j \in [n] \setminus \{k\}$ intersects $P_k'$ only in $A$ since otherwise there is a directed $A$--$v$~path in $P_j \cup (P_k' - e_k) \subseteq
		D - e_k$ as $P_j \subseteq D -e_k$.
		Additionally, if $P_k$ is disjoint to $\bigcup_{j \in [n] \setminus \{k\}} P_j$, then $P_k'$ is disjoint to $\bigcup_{j \in [n] \setminus \{k\}} P_j$.
		Thus the family consisting of $(P_i: i \in [n] \setminus \{k\})$ and $P_k'$ contradicts the minimality of $(P_i)_{i \in [n]}$.
	\end{claimproof}
	Let $v_i \in U$ be the last vertex of $P_i$ and let $Q_i$ be some directed $v_i$--$A$~path that is laced with $P_i$ for every $i \in [n]$, which exists by \cref{prop:laced_paths}.
	Further let $z_i$ be the last vertex of $Q_i$ that is contained in $\Tilde{P}_i$.
	
	\begin{claim}\label{clm:Pavoids}
		The directed path $Q_i$ avoids $\Tilde{P}_j \cup Q_j z_j$ for every $i \neq j \in [n]$.
	\end{claim}
	\begin{claimproof}
		Suppose for a contradiction that $Q_i$ intersects $\Tilde{P}_j \cup Q_j z_j$.
		Then $Q_i \cup \Tilde{P}_j \cup Q_j z_j$ contains a directed $v_i$--$v_j$~path that avoids $e_j$.
		Since $P_i$ avoids $e_j \in E(P_j)$, there exists a directed $A$--$v_j$~path in $(P_i \cup Q_i) \cup (\Tilde{P}_j \cup Q_j z_j) \subseteq D - e_j$, a contradiction to~\cref{clm:notAvoidingFirstEdge}.
	\end{claimproof}
	
	By \cref{clm:Pavoids}, $Q_i$ and $Q_j$ intersect only in $(V(z_iQ_i) \cap V(z_jQ_j)) \setminus \{z_i, z_j\}$ for every $i \neq j \in [n]$.
	In particular, if $x \in Q_i \cap Q_j$ for some $i \neq j \in [n]$, then $xQ_i$ avoids $\bigcup_{k \in [n]} \Tilde{P}_k$.
	Up to re-routing of the directed paths $(z_iQ_i)_{i \in [n]}$ we can assume that $\bigcup_{i \in [n]} (z_iQ_i)$ is a disjoint union of in-arborescences.
	See \cref{fig:out_in_path}.
	
	Finally we obtain the desired butterfly minor $D'$ by
	\begin{itemize}
		\item butterfly contracting $\Tilde{P}_i \cup Q_i z_i$ to a double path $F_i$ for every $i \in [n]$,
		\item butterfly contracting the common terminal segment of $Q_i$ and $Q_j$ to a single vertex for every $i \neq j \in [n]$, and
		\item butterfly contracting the remaining part of $z_i Q_i$ to an edge $(z_i, b_i)$ for every $i \in [n]$.
	\end{itemize}
	If $A$ is of type~\labelcref{itm:main_structure_1}, $D'$ is of type~\labelcref{itm:centre_1} and if $A$ is of type~\labelcref{itm:main_structure_2}, $D'$ is of type~\labelcref{itm:centre_2}.
	Further, if $A$ is of type~\labelcref{itm:main_structure_3}, $D'$ is of type~\labelcref{itm:centre_3}.
	This completes the proof.
\end{proof}

\begin{figure}[ht]
	\begin{tikzpicture}
		
		\node[draw, circle, minimum size=2cm, fill=lightgray] at (2,1) {$A$};
		
		\foreach \a in {3,...,8}{
			\draw[edge][PastelOrange] ({\a+0.1},1) to ({\a+0.9},1);
			\draw[black] (\a,1) node {\huge .};
			\draw[black] ({\a+1},1) node {\huge .};
		}
		
		\draw[edge][CornflowerBlue] (4,1.1) to[bend right] (2.1,3);
		\draw[edge][CornflowerBlue] (8,1.1) to[bend right] (5,1.1);
		\draw[edge][CornflowerBlue] (6,0.9) to[bend left] (4,0.9);
		\draw[edge][CornflowerBlue] (9,0.9) to[bend left] (7,0.9);
		\draw[edge][CornflowerBlue] (0,1.1) to[bend left] (1.9,3);
		\draw[edge][CornflowerBlue] (2,2.9) to (2,2);

		\draw (1.8,2.2) node {$r$};
		
		\draw[CornflowerBlue] (8,2) node {${Q}_2$};
		\draw[CornflowerBlue] (-0.3,2) node {${Q}_1$};
		\draw[PastelOrange] (7,0.5) node {$P_2$};
		\draw[PastelOrange] (-0.3,0.5) node {$P_1$};
		
		\draw[black] (2,2) node {\huge.};
		\draw[black] (0,1) node {\huge.};
		\draw[black] (-1,1) node {\huge.};
		\draw[black] (2,3) node {\huge.};
		\draw[black] (1,1) node {\huge.};
		
		\draw[edge][PastelOrange] (0.9,1) to (0.1,1);
		\draw[edge][PastelOrange, bend left] (-0.1,1) to (-0.9,1);
		\draw[edge, CornflowerBlue, bend left] (-0.9,1) to (-0.1,1);
		
		\draw[edge, CornflowerBlue, dashed] (5.2,1) to (5.9,1);
		\draw[edge, CornflowerBlue, dashed] (7.2,1) to (7.9,1);

	\end{tikzpicture}
	\caption{The setup in the proof of~\cref{propgroundwork}: A strongly connected subgraph $A$ and two disjoint directed $A$--$U$~paths $P_1$ and $P_2$. The directed paths $Q_1$ and $Q_2$ are laced with the directed paths $P_1$ and $P_2$, respectively, and share a common terminal segment. Furthermore, $Q_1$ avoids $P_2$ and $Q_2$ avoids $P_1$.
	}
	\label{fig:out_in_path}
\end{figure}

\section{Proof of \cref{finite}} \label{sec:proof}
We turn our attention to the main theorem:
\finite*

Before proving \cref{finite}, we show the existence of certain in-arborescences and out-arborescences in every subgraph $A$ of type~\labelcref{itm:main_structure_3}, and prove that for every three vertices in a strongly connected directed graph there exists a simple butterfly minor connecting them.

\begin{prop}\label{propinoutbranchingtree}
	Let $(C_j)_{j \in [p]}$ be a sequence of directed cycles with the property that for every $j \neq k$ $C_j\cap C_k$ is a directed path if $|j-k|=1$ and empty otherwise.
	Then for every $k \in [p-1]$ there exist an in-arborescence $T^{\mathrm{in}} \subseteq \bigcup_{j \in [p]} C_j$ and an out-arborescence $T^{\mathrm{out}} \subseteq \bigcup_{j \in [p]} C_j$ that intersect only in their common root such that $\bigcup_{i \in [p] \setminus [k]} V(C_i) \subseteq V(T^{\mathrm{in}})$ and $\bigcup_{i \in [k]} V(C_i) \backslash V(C_{k+1}) \subseteq V(T^{\mathrm{out}})$.
	\end{prop}
	\begin{proof}
		Let $P$ be the directed path in $C_{k+1}\cap C_k$ and let $v$ the last vertex of $P$.
		Since $\bigcup_{j \in [p] \setminus [k]} C_j$ is strongly connected, there exists an in-arborescence spanning $\bigcup_{j \in [p] \setminus [k]} V(C_j)$ with root $v$.
		Since every vertex of $\bigcup_{i \in [k]} V(C_i) \backslash V(C_{k+1})$ can be reached from $v$ in $\bigcup_{j \in [k]} C_j - (V(P) \setminus \{v\})$, there exists an out-arborescence $T^{\mathrm{out}}$ spanning $\bigcup_{i \in [k]} V(C_i) \backslash V(C_{k+1})$ rooted at $v$ that hits $T^{\mathrm{in}}$ only in its root.
		Then $T^{\mathrm{in}}$ and $T^{\mathrm{out}}$ are as desired.
	\end{proof}

\begin{prop}\label{prop3star}
	Let $D$ be a strongly connected directed graph and let $u_1, u_2, u_3 $ be distinct vertices of $D$.
	Then there exists a butterfly minor of $D$ that is
	\begin{itemize}
		\item a double path containing $u_1, u_2, u_3$,
		\item the union of three non-trivial double paths that intersect only in a common endpoint and whose other endpoints are $u_1, u_2, u_3$, or
		\item the union of a directed cycle $C$ of length $3$ and three disjoint (possibly trivial) double paths having one endpoint in $V(C)$ and whose other endpoints are $u_1, u_2, u_3$.
	\end{itemize}
\end{prop}

\begin{proof}
	Let $D'$ be an edge-minimal butterfly minor of $D$ that is strongly connected and contains $u_1, u_2, u_3$.
	By \cref{prop:laced_paths}, there exist a directed $u_1$--$u_2$~path $P$ and a directed $u_2$--$u_1$~path $Q$ in $D'$ that are laced.
	Further, there exist a directed $u_3$--$(P \cup Q)$~path $R$ and a directed $(P \cup Q)$--$u_3$~path $S$ such that $R$ and $S$ are laced by \cref{prop:laced_paths}. 
	
	Note that $D' = P \cup Q \cup R \cup S$ by edge-minimality.
	Further, note that every vertex of $V(D') \setminus \{u_1, u_2, u_3\}$ has in-degree at least $2$ and out-degree at least $2$ since otherwise some edge incident to it can be butterfly contracted, a contradiction to edge-minimality of $D'$.
	
	First, we assume that $P \cup Q$ is a double path.
	Note that $R \cup S$ contains a trail connecting the startvertex of $S$ and the endvertex of $R$.
	Thus the startvertex of $S$ and the endvertex of $R$ coincide since otherwise some edge of $P \cup Q$ with tail in the startvertex of $S$ can be deleted without disconnecting $u_1, u_2, u_3$, a contradiction to the edge-minimality of $D'$.
	This implies that $R \cup S$ is a double path as all vertices of $R \cup S$ but $u_3$ have in- and out-degree at least $2$ in $D'$.
	Therefore $D'$ is a double path containing $u_1, u_2, u_3$, if the startvertex of $S$ and the endvertex of $R$ is either $u_1$ or $u_2$, and otherwise $D'$ is a union of three double paths intersecting only in a common endpoint whose other endpoints are $u_1, u_2, u_3$, as desired.
	
	Now we assume that $P \cup Q$ is not a double path.
	Since all vertices of $P \cup Q$ but $u_1, u_2, u_3$ have in- and out-degree at least $2$ in $D'$, $P \cup Q$ is obtained from a double path with endpoints $u_1, u_2$ by subdividing one edge and the subdivision vertex of $P \cup Q$ is the startvertex of $S$ and the endvertex of $R$.
	This implies that $R \cup S$ is a double path.
	Then $D'$ is the union of a directed cycle $C$ of length $3$ and disjoint (possibly trivial) double paths $F_1, F_2, F_3$ having one endpoint in $V(C)$ and whose other endpoints are $u_1, u_2, u_3$.
	This completes the proof.
\end{proof}

\begin{proof}[Proof of \cref{finite}]
		Let $n \in \NN$ be arbitrary. Further, let $m, m', m'' \in \NN$ with $m \gg m' \gg m'' \gg n$. We set $g(n):=\alpha(m)$, where $\alpha$ is the function from~\cref{propgroundwork}.
		Let $U$ be an arbitrary set of vertices in a strongly connected directed graph $D$ with $|U| \geq g(n)$.
	We apply \cref{propgroundwork} to $D$ and $U$ to obtain a butterfly minor $D'$ as stated in~\cref{propgroundwork} with $m$ elements in $U$.
	\begin{description}
		\item[$D'$ is of type~\labelcref{itm:centre_4}]
		If some cycle directed $C$ contains $n$ elements of $U$, we restrict to $C$ and butterfly contract some edges of $C$ to obtain a directed cycle containing precisely these $n$ vertices, which provides the desired butterfly minor shaped by a star.
		
		Otherwise, there is a family of directed cycles $(C_j)_{j \in [p]}$ with the property that for every $j \neq k$ $C_j\cap C_k$ is a directed path if $|j-k|=1$ and there is an $m$-element subset $U' \subseteq \bigcup_{j \in [p]} V(C_j) \cap U$ such that each directed cycle of $(C_j)_{j \in [p]}$ contains at most one element of $U'$.
			Then there is an $n$-element subset $U'' \subseteq U'$ such that either each element of $U''$ is contained in the intersection of two elements of $(C_j)_{j \in [p]}$ or each element of $U''$ is contained in precisely one element of $(C_j)_{j \in [p]}$.
			In both cases we construct a butterfly minor shaped by a comb.
			In the former case, the union $\bigcup_{j \in [p]} C_j$ can be butterfly contracted to a double path containing the vertices of $U''$ as in \cref{obs:double_path}. Note that double paths are shaped by a comb. In the latter case, we butterfly contract all directed cycles of $\bigcup_{j \in [p]} C_j$ as in \cref{obs:double_path} but butterfly contract the directed cycles containing a vertex of $U''$ to directed cycles of length $3$. This ensures that the vertices of $U''$ are still present in the butterfly minor. Note that also this graph is shaped by a comb.
		
		\item[$D'$ is of type~\labelcref{itm:centre_1}]
		As $A$ is a single vertex, $D'$ is a union of double paths that intersect only in $A$. Thus $D'$ is shaped by a star by definition.
		\item[$D'$ is of type~\labelcref{itm:centre_2}]
		We set $[x,y]:= V(xCy)$ for every $x, y \in V(C)$.
		Furthermore, we write $x \leq_C y \leq_C z$ if $x,y,z \in V(C)$ appear in this order on $C$.
		For every $i \neq i' \in [m]$ precisely one of the following holds:
		\begin{enumerate}[label=(\arabic*)]
			\item\label{itm:prop_2_1} $[b_i,a_i]\cap [b_{i'},a_{i'}] = \emptyset$
			\item\label{itm:prop_2_2} $([b_i,a_i]\subseteq [b_{i'},a_{i'}]) \lor ([b_{i'},a_{i'}]\subseteq [b_{i},a_{i}])$
			\item\label{itm:prop_2_3} $([b_i,a_i]\cap [b_{i'},a_{i'}] \neq \emptyset) \land ([b_i,a_i] \nsubseteq [b_{i'},a_{i'}]) \land ([b_{i'},a_{i'}]\nsubseteq [b_{i},a_{i}])$
		\end{enumerate}
		By Ramsey's theorem, \cref{ramsey1}, we can assume, up to relabelling, that for every $i \neq i' \in [m']$ the same property applies.
		Further, we assume that $a_1 <_C a_2 <_C \dots <_C a_{m'}$.
		From now on we consider the subgraph $C \cup \bigcup_{i \in [m']} (F_i + (a_i,z_i) + (z_i, b_i))$.
		
		\begin{itemize}
			\item If \labelcref{itm:prop_2_1} applies to every $i \neq j \in [m']$:
			Then $b_1 \leq_C a_1 <_C b_2 \leq_C a_2 <_C b_3 \leq_C a_3 <_C \dots$.
			See~\cref{fig:prop_2_1}.
			
			By butterfly contracting $C[[b_i,a_i]]$ to a single vertex  and butterfly contracting $C[[a_i, b_{i+1}]]$ to a single edge for every $i \in [m']$ ($b_{m'+1}:=b_1$) we obtain a butterfly minor that is shaped by a star with $m'$ teeth in $U$.
			
			\begin{figure}[h]
				\begin{tikzpicture}
					\draw (1,0) -- (8,0);
					\draw (7.1,1) -- (8,1);
					\draw[edge]  (1,1) to (1.9,1);
					
					\foreach \a in {3,5}{
						\draw[edge] ({\a+0.1},1) to ({\a+0.9},1);		
					}
					\foreach \a in {2, 4, 6}{
						\draw[edge][dashed] ({\a+0.1},1) to ({\a+0.9},1);		
					}	
					\draw (1,0) .. controls (0,0) and (0,1) .. (1,1);
					\draw (8,0) .. controls (9,0) and (9,1) .. (8,1);
					\draw[] (2,1) node {\huge.};
					\draw (2,0.5) node {$b_1$};
					\draw[] (3,1) node {\huge.};
					\draw (3,0.5) node {$a_1$};
					\draw[] (4,1) node {\huge.};
					\draw (4,0.5) node {$b_2$}; 
					\draw[] (5,1) node {\huge.};
					\draw (5,0.5) node {$a_2$}; 
					\draw[] (6,1) node {\huge.};
					\draw (6,0.5) node {$b_3$}; 
					\draw[] (7,1) node {\huge.};
					\draw (7,0.5) node {$a_3$};  
					
					\draw[edge][CornflowerBlue] (3,1.1) to (2.6,1.9);
					\draw[CornflowerBlue] (2.5,2) node {\huge.};
					\draw[edge][CornflowerBlue] (2.4,1.9) to (2,1.1);
					
					\draw[edge][PastelOrange] (5,1.1) to (4.6,1.9);
					\draw[PastelOrange] (4.5,2) node {\huge.};
					\draw[edge][PastelOrange] (4.4,1.9) to (4,1.1);
					
					\draw[edge][LavenderMagenta] (7,1.1) to (6.6,1.9);
					\draw[LavenderMagenta] (6.5,2) node {\huge.};
					\draw[edge][LavenderMagenta] (6.4,1.9) to (6,1.1);
					
				\end{tikzpicture}
				\caption{A butterfly minor $D'$ of type~\labelcref{itm:centre_2} where \labelcref{itm:prop_2_1} applies to all pairs of~$[m']$ in the proof of~\cref{finite}.} \label{fig:prop_2_1}
			\end{figure}
			
			\item If \labelcref{itm:prop_2_2} applies to every $i \neq j \in [m']$:
			There exist one interval $[b_i, a_i]$ that is contained in all intervals of this form.
			We assume without loss of generality that $[b_1,a_1]$ is this interval.
			Then $\dots \leq_C b_3 \leq_C b_2 \leq_C b_1 \leq_C a_1 <_C a_2 <_C a_3 <_C \dots$.
			We remark that elements of $(b_i: i \in [n])$ can coincide.
			See~\cref{fig:prop_2_2}.
			
			We construct the desired butterfly minor in the following way:
			Firstly, we delete the edges of $C[[a_{m'},b_{m'}]]$.
			Now every vertex in $[b_{m'},a_1)$ has out-degree precisely $1$ and we butterfly contract $C[[b_{m'},a_1]]$ to the single vertex $a_1$.
			Similarly, each vertex in $(a_1,a_{m'}]$ has in-degree precisely $1$ and we butterfly contract $C[[a_1,a_{m'}]]$ to the single vertex $a_1$.
			This butterfly minor is the union of disjoint double paths $(F_i: i \in [m'])$ each connected to a single vertex $a_1$ via edges $(a_i, z_i)$ and $(z_i, b_i)$.
			Thus this butterfly minor is shaped by a star with $m'$ teeth in $U$.
			
			\begin{figure}
				\centering
				\begin{subfigure}{\textwidth}
					\centering
					\begin{tikzpicture}
						\draw[dotted] (1,0) -- (8,0);
						\draw[dotted] (7.1,1) -- (8,1);
						\draw[edge][dotted]  (1,1) to (1.9,1);
						
						\foreach \a in {2, ..., 6}{
							\draw[edge] ({\a+0.1},1) to ({\a+0.9},1);		
						}		
						\draw[dotted] (1,0) .. controls (0,0) and (0,1) .. (1,1);
						\draw[dotted] (8,0) .. controls (9,0) and (9,1) .. (8,1);
						\draw[] (2,1) node {\huge.};
						\draw (2,0.5) node {$b_3$};
						\draw[] (3,1) node {\huge.};
						\draw (3,0.5) node {$b_2$};
						\draw[] (4,1) node {\huge.};
						\draw (4,0.5) node {$b_1$}; 
						\draw[] (5,1) node {\huge.};
						\draw (5,0.5) node {$a_1$}; 
						\draw[] (6,1) node {\huge.};
						\draw (6,0.5) node {$a_2$}; 
						\draw[] (7,1) node {\huge.};
						\draw (7,0.5) node {$a_3$};  
						
						\draw[edge][CornflowerBlue] (6,1.1) to (4.6,1.9);
						\draw[CornflowerBlue] (4.5,2) node {\huge.};
						\draw[edge][CornflowerBlue] (4.4,1.9) to (3,1.1);
						
						\draw[edge][PastelOrange] (5,1.1) to (4.6,1.4);
						\draw[PastelOrange] (4.5,1.5) node {\huge.};
						\draw[edge][PastelOrange] (4.4,1.4) to (4,1.1);
						
						\draw[edge][LavenderMagenta] (7,1.1) to (4.6,2.4);
						\draw[LavenderMagenta] (4.5,2.5) node {\huge.};
						\draw[edge][LavenderMagenta] (4.4,2.4) to (2,1.1);
					\end{tikzpicture}
					\caption{The original graph.}
					\label{subfig1}
				\end{subfigure}
				\begin{subfigure}{0.55\textwidth}
					\centering
					\begin{tikzpicture}
						\draw[color=gray!30, line width=1,rounded corners=0.2cm, fill=gray!30] (4.6,0.7) rectangle (7.3,1.3);
						
						\draw[color=gray!30, line width=1,rounded corners=0.2cm, fill=gray!30] (1.7,0.7) rectangle (5.4,1.3);
						
						\draw[color=gray!70, line width=1,rounded corners=0.2cm, fill=gray!70] (4.6,0.7) rectangle (5.4,1.3);
						
						\foreach \a in {2, ..., 6}{
							\draw[edge] ({\a+0.1},1) to ({\a+0.9},1);		
						}
						
						\draw[] (2,1) node {\huge.};
						\draw (2,0.2) node {$b_3$};
						\draw[] (3,1) node {\huge.};
						\draw (3,0.2) node {$b_2$};
						\draw[] (4,1) node {\huge.};
						\draw (4,0.2) node {$b_1$}; 
						\draw[] (5,1) node {\huge.};
						\draw (5,0.2) node {$a_1$}; 
						\draw[] (6,1) node {\huge.};
						\draw (6,0.2) node {$a_2$}; 
						\draw[] (7,1) node {\huge.};
						\draw (7,0.2) node {$a_3$};  
						
						\draw[edge][CornflowerBlue] (6,1.4) to (4.6,2.2);
						\draw[CornflowerBlue] (4.5,2.3) node {\huge.};
						\draw[edge][CornflowerBlue] (4.4,2.2) to (3,1.4);
						
						\draw[edge][PastelOrange] (5,1.4) to (4.6,1.7);
						\draw[PastelOrange] (4.5,1.8) node {\huge.};
						\draw[edge][PastelOrange] (4.4,1.7) to (4,1.4);
						
						\draw[edge][LavenderMagenta] (7,1.4) to (4.6,2.7);
						\draw[LavenderMagenta] (4.5,2.8) node {\huge.};
						\draw[edge][LavenderMagenta] (4.4,2.7) to (2,1.4);
					\end{tikzpicture}
					\caption{The graph after deletion of the edges in $C[[a_{m'},b_{m'}]]$. The gray marked edges can be butterfly contracted.}
					\label{subfig2}
				\end{subfigure}
				\hfill
				\begin{subfigure}{0.4\textwidth}
					\centering
					\begin{tikzpicture}
						\draw (4.5,0.2) node {$a_1$}; 
						\draw[] (4.5,1) node {\huge.};
						
						\draw[edge][CornflowerBlue] (4.6,1.1) to[bend right] (4.6,2.2);
						\draw[CornflowerBlue] (4.5,2.3) node {\huge.};
						\draw[edge][CornflowerBlue] (4.4,2.2) to[bend right] (4.4,1.1);
						
						\draw[edge][PastelOrange] (4.6,1.1) to[bend right]  (4.6,1.7);
						\draw[PastelOrange] (4.5,1.8) node {\huge.};
						\draw[edge][PastelOrange] (4.4,1.7) to[bend right]  (4.4,1.1);
						
						\draw[edge][LavenderMagenta] (4.6,1.1) to[bend right]  (4.6,2.7);
						\draw[LavenderMagenta] (4.5,2.8) node {\huge.};
						\draw[edge][LavenderMagenta] (4.4,2.7) to[bend right]  (4.4,1.1);
					\end{tikzpicture}
					\caption{The butterfly minor obtained by butterfly contracting the gray marked edges.}
					\label{subfig3}
				\end{subfigure}
				
				\caption{A butterfly minor $D'$ of type~\labelcref{itm:centre_2} where \labelcref{itm:prop_2_2} applies to all pairs of $[m']$ in the proof of~\cref{finite}.}
				\label{fig:prop_2_2}
			\end{figure}

			\item If \labelcref{itm:prop_2_3} applies to every $i \neq j \in [m']$:
			The intervals $[b_i, a_i], [b_j, a_j]$ intersect either in one or in two intervals.
			By Ramsey's theorem, \cref{ramsey1}, we can assume, up to relabelling, that either $[b_i, a_i] \cap [b_j, a_j]$ is one interval for every $i \neq j \in [2n]$ or $[b_i, a_i] \cap [b_j, a_j]$ are two intervals for every $i \neq j \in [2n]$.
			Note that \labelcref{itm:prop_2_3} implies $a_i \neq b_i$ for every $i \neq j \in [m']$.
			
			If these intervals intersect pairwise in one interval, we can assume, up to relabelling, that the intervals $\{[b_i, a_i]: i \in [n] \setminus \{1\} \}$ contain either $a_1$ but not $b_1$ or $b_1$ but not $a_1$.
			In the former case, $a_1 \leq_C a_i <_C b_1 \leq_C b_i \leq_C a_1$ for every $i \in [n] \setminus \{1\}$.
			We construct the desired butterfly minor in the following way: See~\cref{fig:prop_2_3}.
			Firstly, we consider the subgraph $C \cup \bigcup_{i \in [n]} (F_i + (a_i, z_i) + (z_i, b_i))$ and delete the edges of $C[[a_{n},b_1]]$.
			Note that every vertex in $[b_{1},a_1)$ has out-degree precisely $1$ and we butterfly contract $C[[b_{1},a_1]]$ to the vertex~$a_1$.
			Similarly, each vertex in $(a_1,a_{n}]$ has in-degree precisely $1$ and we butterfly contract $C[[a_1,a_{n}]]$ to the single vertex $a_1$.
			This butterfly minor is the union of disjoint double paths $(F_i: i \in [n])$ that are connected to $a_1$ via edges $(a_i, z_i)$ and $(z_i, b_i)$.
			Thus this butterfly minor is shaped by a star with $n$ teeth in $U$.
			In the latter case, $a_1 <_C b_i \leq_C b_1 \leq_C a_i \leq_C a_1$ for every $i \in [n] \setminus \{1\}$ and we apply a similar argument to obtain a butterfly minor shaped by a star with $n$ teeth in $U$.
			
			If these intervals intersect pairwise in two intervals, then $a_i <_C b_i \leq_C a_j <_C b_j \leq_C a_i$ for every $i \neq j \in [2n]$.
			See~\cref{fig:prop_2_3_2}.
			We consider the directed cycle $C'$ consisting of the directed paths $b_1 C a_{2}, \dots, b_{2n} C a_1$ and the edges $(a_i, z_i), (z_i, a_i)$ for every $i \in [2n]$.
			The union of $C'$ and $(F_i: i \in [2n])$ is a subdivision of a directed graph shaped by a star with $2n$ teeth in $U$, as desired.
		\end{itemize}
		
		\begin{figure}[ht]
			\begin{tikzpicture}
				\draw[] (1,0) -- (8,0);
				\draw[] (7.1,1) -- (8,1);
				\draw[edge][]  (1,1) to (1.9,1);
				
				\foreach \a in {2, 3, 5, 6}{
					\draw[edge] ({\a+0.1},1) to ({\a+0.9},1);		
				}	
				\draw[edge][dotted] (4.1,1) to (4.9,1);
				
				\draw[] (1,0) .. controls (0,0) and (0,1) .. (1,1);
				\draw[] (8,0) .. controls (9,0) and (9,1) .. (8,1);
				\draw[] (2,1) node {\huge.};
				\draw (2,0.5) node {$a_1$};
				\draw[] (3,1) node {\huge.};
				\draw (3,0.5) node {$a_2$};
				\draw[] (4,1) node {\huge.};
				\draw (4,0.5) node {$a_3$}; 
				\draw[] (5,1) node {\huge.};
				\draw (5,0.5) node {$b_1$}; 
				\draw[] (6,1) node {\huge.};
				\draw (6,0.5) node {$b_2$}; 
				\draw[] (7,1) node {\huge.};
				\draw (7,0.5) node {$b_3$};  
				
				\draw[edge][CornflowerBlue] (2,1.1) to (2.9,1.9);
				\draw[CornflowerBlue] (3,2) node {\huge.};
				\draw[edge][CornflowerBlue] (3.1,1.9) to (5,1.1);
				
				\draw[edge][PastelOrange] (3,1.1) to (4.4,1.9);
				\draw[PastelOrange] (4.4,2) node {\huge.};
				\draw[edge][PastelOrange] (4.6,1.9) to (6,1.1);
				
				\draw[edge][LavenderMagenta] (4,1.1) to (5.9,1.9);
				\draw[LavenderMagenta] (6,2) node {\huge.};
				\draw[edge][LavenderMagenta] (6.1,1.9) to (7,1.1);
				
			\end{tikzpicture}
			\caption{A butterfly minor $D'$ of type~\labelcref{itm:centre_2} where \labelcref{itm:prop_2_3} applies to all pairs of $[m']$ and the intervals intersect pairwise in one interval in the proof of~\cref{finite}.}
			\label{fig:prop_2_3}
		\end{figure}
		
		\begin{figure}[h]
			\begin{tikzpicture}
				\draw (1,0) -- (8,0);
				\draw (7.1,1) -- (8,1);
				\draw[edge]  (1,1) to (1.9,1);
				
				\foreach \a in {3,5}{
					\draw[edge] ({\a+0.1},1) to ({\a+0.9},1);		
				}
				\foreach \a in {2, 4, 6}{
					\draw[edge][dotted] ({\a+0.1},1) to ({\a+0.9},1);		
				}	
				\draw (1,0) .. controls (0,0) and (0,1) .. (1,1);
				\draw (8,0) .. controls (9,0) and (9,1) .. (8,1);
				\draw[] (2,1) node {\huge.};
				\draw (2,0.5) node {$a_1$};
				\draw[] (3,1) node {\huge.};
				\draw (3,0.5) node {$b_1$};
				\draw[] (4,1) node {\huge.};
				\draw (4,0.5) node {$a_2$}; 
				\draw[] (5,1) node {\huge.};
				\draw (5,0.5) node {$b_2$}; 
				\draw[] (6,1) node {\huge.};
				\draw (6,0.5) node {$a_3$}; 
				\draw[] (7,1) node {\huge.};
				\draw (7,0.5) node {$b_3$};  
				
				\draw[edge][CornflowerBlue] (2.6,1.9) to (3,1.1);
				\draw[CornflowerBlue] (2.5,2) node {\huge.};
				\draw[edge][CornflowerBlue] (2,1.1) to (2.4,1.9);
				
				\draw[edge][PastelOrange] (4.6,1.9) to (5,1.1);
				\draw[PastelOrange] (4.5,2) node {\huge.};
				\draw[edge][PastelOrange] (4,1.1) to (4.4,1.9);
				
				\draw[edge][LavenderMagenta] (6.6,1.9) to (7,1.1);
				\draw[LavenderMagenta] (6.5,2) node {\huge.};
				\draw[edge][LavenderMagenta] (6,1.1) to (6.4,1.9);
				
			\end{tikzpicture}
			\caption{A butterfly minor $D'$ of type~\labelcref{itm:centre_2} where \labelcref{itm:prop_2_3} applies to all pairs of $[m']$ and the intervals intersect pairwise in two intervals in the proof of~\cref{finite}.} \label{fig:prop_2_3_2}
		\end{figure}
		
		\item[$D'$ is of type~\labelcref{itm:centre_3}]
		For every $v \in \bigcup_{j \in [p]} V(C_j)$ let $\gamma(v) \in [p]$ be the minimal index with $v \in V(C_{\gamma(v)})$.
		Then either $\gamma(a_i)\leq \gamma (b_i)$ or $\gamma(b_i)\leq \gamma (a_i)$.
		By Ramsey's theorem, \cref{ramsey1}, we can assume, up to relabelling of the elements in $[m]$ and up to reversing the labelling of the directed cycles $(C_j)_{j \in [p]}$, that $\gamma(a_i)\leq \gamma (b_i)$ for every $i \in [m']$.
		
		As provided by \cref{propgroundwork}, $\gamma(a_i)\neq \gamma(a_{i'})$ holds for every $i \neq i' \in [m']$.
		We assume, up to relabelling of $[m']$, that $\gamma(a_i) < \gamma(a_{i'})$ for every $i < i' \in [m']$.

		For every $i < i' \in [m']$ either $\gamma(b_i) < \gamma(a_{i'})$ or $\gamma(b_i) \geq \gamma(a_{i'})$ holds.
		By Ramsey's theorem, \cref{ramsey1}, we can assume, up to relabelling, that one property holds for every pair of distinct elements in $[m'']$.
		
		\begin{itemize}
			\item In the former case, i.e.\ $\gamma(b_i) < \gamma(a_{i'})$ for every $i < i' \in [m'']$: See~\cref{fig:centre_3_1}.
			Then $
			\gamma(a_1) \leq \gamma(b_1) < \gamma(a_2) \leq \gamma(b_2) < \dots$.
			We consider the subgraph
			\begin{equation*}
				\bigcup_{j \in [p]} C_j \cup \bigcup_{i \in [m''] \cap 2\NN} (F_i + (a_i, z_i) + (z_i, b_i))
			\end{equation*}
			and butterfly contract $C_{\gamma(a_i)-1}\cap C_{\gamma(a_i)}$ to a single vertex $z_i$ for every even $i \in [m'']$.
			Note that $C_{\gamma(a_i)-1}\cap C_{\gamma(a_i)}$ neither contains $a_i$ nor $b_i$ for every even $i \in [m'']$.
			
			For even $i \in [m'']$, the subgraph $H_i:=F_i \cup \bigcup_{\gamma(a_i) \leq j < \gamma(a_{i+2})} C_j$ is strongly connected and contains $z_i$, $z_{i+2}$ and $v_i$.
			We apply~\cref{prop3star} to obtain a butterfly minor $D_{i}$ of $H_i$ that is a double path containing $z_i, z_{i+2}, v_i$, the union of non-trivial double paths that intersect only in a common endpoint and whose other endpoints are $z_i, z_{i+2}, v_i$, or the union of a directed cycle $C$ of length $3$ and disjoint (possibly trivial) double paths having one endpoint in $V(C)$ and whose other endpoints are $z_i, z_{i+2}, v_i$.
			
			Then there exists an $n$-element subset $J \subseteq [m'']$ such that all $(D_j)_{j \in J}$ are of the same type and the elements of $J$ differ pairwise by at least two.
			We set $D_i':= D_i$ for every $i \in J$ and let $D_i'$ be the butterfly minor of $D_i$ that is a double path with endpoints $z_i, z_{i+2}$ for every $i \in [m''] \setminus J$.
			Then $\bigcup_{i \in [m'']} D_i'$ the desired butterfly minor shaped by a comb with $n$ teeth in $U$.
			
			\begin{figure}[ht]
				\begin{tikzpicture}

					\foreach \a in {1,3,5,7,9}{
						\draw[edge] ({\a+0.05},0) to ({\a+0.95},0);
						\draw[edge] ({\a+1},0.05) to ({\a+1},0.95);
						\draw[edge] ({\a+0.95},1) to ({\a+0.05},1);	
						\draw[edge] (\a,0.95) to ({\a},0.05);
						
						\draw[edge] ({\a+2},0) to ({\a+1},0);
						\draw[edge] ({\a+1},1) to ({\a+2},1);
					}		
					
					\draw[edge][PastelOrange] ({1*4-2},1.1) to ({1*4-1.1},2);
					\draw[edge][PastelOrange] ({1*4-1},1.9) to ({1*4-1},1.1);	
					\draw[edge][LavenderMagenta] ({4-3},0.95) to ({1*4-3},0.05);
					\draw[PastelOrange] ({1*4-1},2) node {\huge.};
					\draw[PastelOrange] ({4-1.4},2.2) node {$v_1$};
					
					\draw[edge][CornflowerBlue] ({2*4-2},1.1) to ({2*4-1.1},2);
					\draw[edge][CornflowerBlue] ({2*4-1},1.9) to ({2*4-1},1.1);	
					\draw[edge][LavenderMagenta] ({2*4-3},0.95) to ({2*4-3},0.05);
					\draw[CornflowerBlue] ({2*4-1},2) node {\huge.};
					\draw[CornflowerBlue] ({8-1.4},2.2) node {$v_2$};
					
					\draw[edge][LavenderMagenta] ({3*4-2},1.1) to ({3*4-1.1},2);
					\draw[edge][LavenderMagenta] ({3*4-1},1.9) to ({3*4-1},1.1);	
					\draw[edge][LavenderMagenta] ({3*4-3},0.95) to ({3*4-3},0.05);
					\draw[LavenderMagenta] ({3*4-1},2) node {\huge.};
					\draw[LavenderMagenta] ({12-1.4},2.2) node {$v_1$};

					\foreach \a in {2, 4, 6}{	
						\draw[LavenderMagenta] ({\a*2-3},-0.5) node {$z_\a$};
						
					}
					\draw[edge] (11,0.95) to (11,0.05);
				\end{tikzpicture}
				\caption{A butterfly minor $D'$ of type~\labelcref{itm:centre_3} where $\gamma(b_i) < \gamma(a_{i'})$ for every $i < i' \in [m'']$ in the proof of~\cref{finite}.} \label{fig:centre_3_1}
			\end{figure}
			
			\item In the latter case, i.e.\ $ \gamma(b_i) \geq \gamma(a_{i'})$ for every $i < i' \in [m'']$: See~\cref{fig:centre_3_2}. Then $\gamma(a_{m''}) \leq \gamma(b_i)$
			for every $i \in [m'']$ and thus $\gamma(a_{m''-1}) < \gamma(b_i)$
			for every $i \in [m'']$.
			In particular, all $b_i$ for $i \in [m''-1]$ are contained in $\bigcup_{i > \gamma(a_{m''-1})} V(C_i)$ and all $a_i$ for $i \in [m''-1]$ are contained in  $\bigcup_{i \leq \gamma(a_{m''-1})} V(C_i)$.
			
			\begin{figure}[ht]
				\begin{tikzpicture}

					\foreach \a in {1,3,5,7,9}{
						\draw[edge] ({\a+0.05},0) to ({\a+0.95},0);
						\draw[edge] ({\a+1},0.05) to ({\a+1},0.95);
						\draw[edge] ({\a+0.95},1) to ({\a+0.05},1);	
						\draw[edge] (\a,0.95) to ({\a},0.05);
						
						\draw[edge] ({\a+2},0) to ({\a+1},0);
						\draw[edge] ({\a+1},1) to ({\a+2},1);
					}		
					
					\draw[edge][PastelOrange] ({1+1},1.1) to ({1+3.9},2);
					\draw[edge][PastelOrange] ({1+4.1},2) to ({1+7},1.1);
					\draw[PastelOrange] (5,2) node {\huge.};	
					\draw[PastelOrange] (5,2.4) node {$v_1$};
					
					\draw[edge][CornflowerBlue] ({2+1},1.1) to ({2+3.9},2);
					\draw[edge][CornflowerBlue] ({2+4.1},2) to ({2+7},1.1);
					\draw[CornflowerBlue] (6,2) node {\huge.};	
					\draw[CornflowerBlue] (6,2.4) node {$v_2$};
					
					\draw[edge][LavenderMagenta] ({3+1},1.1) to ({3+3.9},2);
					\draw[edge][LavenderMagenta] ({3+4.1},2) to ({3+7},1.1);
					\draw[LavenderMagenta] (7,2) node {\huge.};	
					\draw[LavenderMagenta] (7,2.4) node {$v_3$};

					\draw[edge] (11,0.95) to (11,0.05);
				\end{tikzpicture}
				\caption{A butterfly minor $D'$ of type~\labelcref{itm:centre_3} where $\gamma(b_i) \geq \gamma(a_{i'})$ for every $i < i' \in [m'']$ in the proof of~\cref{finite}.} \label{fig:centre_3_2}
			\end{figure}
			
			By \cref{propinoutbranchingtree}, there are $T^{\text{in}}$ and $T^{\text{out}}$ as stated such that all $b_i$ are in $V(T^{\text{in}})$ and all $a_i$ are in $V(T^{\text{out}})$ for $i \in [m''-1]$.
			We consider the subgraph
			\begin{equation*}
				T^{\text{in}} \cup T^{\text{out}} \cup \bigcup_{i \in [m''-1]} (F_i + (a_i, z_i) + (z_i, b_i))
			\end{equation*}
			and butterfly contract $T^{\text{in}}$ and $T^{\text{out}}$ to their common root.
			This butterfly minor is shaped by a star with $m''-1$ teeth in $U$.
		\end{itemize}
	\end{description}
	This completes the proof.
\end{proof}

\bibliography{ref.bib}

\end{document}